\documentclass[12pt]{amsart}
\usepackage{amsmath,amssymb,amsthm,verbatim,hyperref,xcolor}

\DeclareMathOperator{\val}{val}

\DeclareFontFamily{U}{wncy}{}
    \DeclareFontShape{U}{wncy}{m}{n}{<->wncyr10}{}
    \DeclareSymbolFont{mcy}{U}{wncy}{m}{n}
    \DeclareMathSymbol{\Sha}{\mathord}{mcy}{"58} 

\begin{document}
\newtheorem{thm}{Theorem}[section]
\newtheorem{lem}[thm]{Lemma}
\newtheorem{dfn}[thm]{Definition}
\newtheorem{propdefn}[thm]{Proposition-Definition}
\newtheorem{cor}[thm]{Corollary}
\newtheorem{conj}[thm]{Conjecture}
\newtheorem{clm}[thm]{Claim}
\theoremstyle{remark}
\newtheorem{exm}[thm]{Example}
\newtheorem{rem}[thm]{Remark}
\def\N{{\mathbb N}}
\def\F{{\mathbb F}}
\def\G{{\mathbb G}}
\def\A{{\mathbb A}}
\def\Q{{\mathbb Q}}
\def\R{{\mathbb R}}
\def\C{{\mathbb C}}
\def\P{{\mathbb P}}
\def\Z{{\mathbb Z}}
\def\v{{\mathbf v}}
\def\p{{\mathfrak p}}
\def\x{{\mathbf x}}
\def\O{{\mathcal O}}
\def\M{{\mathcal M}}
\def\kbar{{\bar{k}}}
\def\tr{\mbox{Tr}}
\def\id{\mbox{id}}
\def\qr#1#2{\left(\frac{#1}{#2}\right)}
\renewcommand{\qedsymbol}{{\tiny $\clubsuit$ \normalsize}}
\def\qed{{\tiny $\clubsuit$ \normalsize}}

\renewcommand{\theenumi}{\alph{enumi}}

\title{Modular curves and bad reduction}

\author{Adam Logan}

\email{adam.m.logan@gmail.com}

\author{David McKinnon}

\email{dmckinnon@uwaterloo.ca}

\begin{abstract}
    We prove results that imply, under various hypotheses, that every elliptic curve over a 
    number field $k$ corresponding to a point on a modular curve has bad reduction at a
    certain prime $\p$ of $\O_k$.
    For example, every elliptic curve with a cyclic torsion subgroup of order 20 defined over 
    $\Q(\sqrt{-11})$ or $\Q(\sqrt{17})$ has bad reduction at all primes lying over $3$.  
    The proofs of these statements are quite different, since $3$ is split in $\Q(\sqrt{-11})$
    and inert in $\Q(\sqrt{17})$.
\end{abstract}

\maketitle

\section{Introduction}

At first glance, the structure of the rational torsion of an elliptic curve and the set of places of bad reduction of that elliptic curve seem to have nothing to do with one another.  And yet, all of the infinitely many isomorphism classes of elliptic curves over $\Q(\sqrt{-11})$ that have a rational subgroup of order 20 have bad reduction at every place lying over $3$.  

The reason why has to do with modularity, and hinges crucially on the fact that the modular curve $X_0(20)$ has genus $1$.  The principles involved apply more widely than the specific scenario from the previous paragraph, but they point at curious yet simple phenomena that connect otherwise disparate features of the curves.

Most of the existing results connecting bad reduction with torsion structure involve imposing restrictions on the torsion based on places of bad reduction.  For example, if an elliptic curve over a number field has additive reduction at a place of characteristic $p$, then the rational torsion of the curve is a finite $p$-group.  (See for example Proposition VII.3.1 from \cite{Si}.)  Work of Yasuda (\cite{Ya1}, \cite{Ya2}) also deduces the absence of certain torsion structures from information on the locus of bad reduction.

In this work, by contrast, we use information about the torsion structure to deduce the existence of places of bad reduction.  We document various settings, both general and specific, in which a feature of the torsion of the curve forces bad reduction at a certain set of places.  Indeed, the reduction is not merely bad, but in fact is not even potentially good.  Our techniques require the torsion feature to correspond to a modular curve of genus one, but they are sometimes applicable to 
the higher genus case as well.  The greatest drawback of the higher genus case, of course, is the finiteness of rational points on a curve of genus at least $2$: for each number field $K$ the
results apply to only finitely many isomorphism classes of elliptic curves over $K$.

\section{Main results}


We begin with an example of how our techniques can be used.  This example readily generalizes.

\begin{thm}\label{rankone}
	Let $k$ be a number field such that the modular curve $C=X_0(20)$ has Mordell-Weil rank $n$, where we choose some cusp of $C$ to be the identity element of $C$.  Assume that the $k$-rational torsion of $C$ has order $6$---that is, extending scalars to $k$ does not increase the torsion.
	
	Let $P_1, \dots, P_n$ be generators of the free part of $C(k)$, and let the $E_i$ be the elliptic curves corresponding to the $P_i$.  Let $\ell$ be a prime of $\O_k$ with $\gcd(\ell,20)=1$.  If every $E$ does not have potentially good reduction at $\ell$, then every $k$-rational elliptic curve with a $k$-rational cyclic torsion subgroup of order 20 does not have potentially good reduction at $\ell$.
\end{thm}

There are many examples to which this theorem applies, of which one ($\Q(\sqrt{-11}$) is described in the introduction.  There are examples in higher rank, however.  For example, let $k = \Q(\sqrt{142})$ and let $\ell$ be a prime above one of $3, 7, 31, 71$.  Then $X_0(20)$ has rank $2$ over $k$ and the generators both give curves whose $j$-invariant is not $\ell$-integral.  There are even examples of rank $3$,  such as $k = \Q(\sqrt{4966}), \ell = \p_{191}$ or $k = \Q(\sqrt{12837}), \ell = \p_{11}$ (ramified in both cases).

\begin{proof}
    The modular curve $C=X_0(20)$ has six cusps and genus $1$.  If we choose one of these cusps to be the identity element in the Mordell-Weil group law, then the other cusps become torsion points.  Indeed, the entire Mordell-Weil group of $C$ over $\Q$ is of order $6$, consisting of exactly these cuspidal points.

The Mordell-Weil group of $C$ over $k$ has rank $n$ and torsion of order~$6$.  
The hypothesis that the $E_i$ do not have potentially good reduction at $\ell$ implies that
the generators $P_i$ of this group meets the locus at infinity modulo $\ell$.  The subset
$S$ of $C(k)$ consisting of points that meet the locus at infinity modulo $\ell$ is closed
under addition and negation and contains $C(\Q)$.  Therefore it is a subgroup; by hypothesis
it contains all the $P_i$ and the torsion subgroup of $C(k)$, so it is all of
$C(k)$.  In particular, every $P \in C(k)$ meets the locus at infinity modulo $\ell$.

It follows that the $j$-invariant of every curve $E$ corresponding to a $P \in C_k$ is congruent
mod $\ell$ to that of a curve corresponding to a cusp, which is $\infty$.  In other words,
$\val_\ell j(E) < 0$, which is equivalent to not having potentially good reduction at $\ell$.
\end{proof}

\begin{thm}\label{general}
	Let $E$ be an elliptic curve with a rational torsion subgroup of order $20$ defined over a number field $k$.  If $3$ is totally split in $k$, then $E$ has bad reduction at any prime lying over $3$.
\end{thm}

\begin{proof} The modular curve $C=X_0(20)$ has six cusps and genus one.  If we choose one of these cusps to be the identity element in the Mordell-Weil group law, then the other cusps become torsion points.  Indeed, the entire Mordell-Weil group of $C$ over $\Q$ is of order six, consisting of exactly these cuspidal points.

The Mordell-Weil group of $C$ over $k$ is generated by points defined over some quadratic number field $K/\Q$.  Modulo any prime $\ell$ of $k$ lying over $3$, $C$ has six rational points.  Each of these is the reduction of a cusp.  Therefore any $K$-rational point is congruent to some cusp modulo $\ell$, and so every $K$-rational curve parameterized by $C$ has bad reduction at $\ell$.  
\end{proof}

Theorem~\ref{general} explains the example given in the introduction.  The field $K=\Q(\sqrt{-11})$ is split at $3$, and the genus one curve $X_0(20)$ has infinitely many $K$-rational points.  

Let us now state and prove more general results.  After this we will discuss
the problem of constructing examples in which the hypotheses are verified.

\begin{thm}\label{thm:weaker}
  Let $E$ be a modular curve of genus $1$, let $p$ be
  a prime of good reduction for $E$, and let $q$ be a power of $p$.
  Suppose that every $\F_q$-point of $E$
  has the same reduction as a cusp (rational or not).  Let $K$ be a number
  field and $\p$ a place above $p$ with residue field contained in $\F_q$.
  Then every elliptic curve over $K$ represented by a $K$-point of $E$ does not
  have potentially good reduction at $\p$.
\end{thm}

\begin{proof} This is really the same argument as before.  Let $P$ be a
  $K$-point of $E$.  In the reduction of $E(K)$ mod $\p$, the image of
  $P$ must be congruent to a cusp, and so the $j$-invariant of the point
  corresponding to $P$ is not integral at $\p$.  The conclusion follows
  from this.
\end{proof}

In the special case of a quadratic field, we can sometimes do business in
cases where not all points are represented by cusps.

\begin{thm}\label{thm:stronger}
  Let $E$ be a modular elliptic curve as before, let $K$ be a quadratic
  field, let $p$ be a prime inert in $K$, and suppose that every $\F_p$-point
  of $E$ is represented by a cusp and that the map $E \to \P^1$ to the
  $x$-line is surjective on $\F_p$-points.  Let $E^\sigma$ be the twist of
  $E$ corresponding to $K$.  Suppose that the natural map
  $E(\Q) \oplus E^\sigma(\Q) \to E(K)$ is surjective.  Then every elliptic
  curve over $K$ represented by a $K$-point of $E$ does not have potentially
  good reduction at $p$.
\end{thm}

\begin{proof} Let $P \in E(K)$ and write $P = C + R$, where $C \in E(\Q)$
  and $R \in E^\sigma(\Q)$.  Let $R_0$ be the
  image of $R$ in the $x$-line: then $R_0$ reduced mod $p$ is also the image of
  a $\Q$-point of $E$, since the reduction map is surjective on $\F_p$-points.
  This is only possible if $R_0$ is a Weierstrass point mod $p$, so 
  $R \bmod p \in E(\F_p)$ and hence it is the reduction of a point of $E(\Q)$.
  Thus $P \bmod p$ is the reduction of a point of $E(\Q)$ as well; by hypothesis
  it is the reduction of a cusp.  The argument proceeds as before.
\end{proof}

\begin{exm}
  We now consider the examples of modular curves $X_0(N)$ of genus $1$.
  For $N = 11, 14, 15, 17, 19, 21, 27$ there are non-cuspidal rational points.
  Since the reduction of torsion points at a prime of good reduction is always
  injective, and these curves are of rank $0$, it is not possible for every
  mod $p$ point to be represented by a cusp for any such $p$.  In particular, our techniques do not apply to these curves.

  For the remaining $N = 20, 24, 32, 36, 49$ it is straightforward to check
  the conditions.  We find that the hypotheses of Theorems
  \ref{thm:weaker},~\ref{thm:stronger}
  are satisfied as indicated in Table~\ref{table:x0}.

  \begin{table}[h]
    \caption{Levels and primes for which the hypotheses of our theorems hold for
      $X_0(N)$.}
    \begin{tabular}{|c|c|c|}
      \hline
      $N$&\ref{thm:weaker}&\ref{thm:stronger}\\ \hline
      $20$&3,7&3\\ \hline
      $24$&5,7,11&5\\ \hline
      $32$&3,5&3\\ \hline
      $36$&5,7&---\\ \hline
      $49$&2&---\\ \hline
    \end{tabular}\label{table:x0}
  \end{table}
\end{exm}

So, for example, from the second line of the table we see that any elliptic curve over $K=\Q(\sqrt{10})$ that has a $K$-rational subgroup of order 24 must also have bad reduction at $5$.  This is because $5$ is ramified in $\Q(\sqrt{10})$, allowing the application of Theorem~\ref{thm:weaker}.  Note that this is a statement about infinitely many isomorphism classes of curves, since the rank of $X_0(24)$ over $\Q(\sqrt{10})$ is positive.

\begin{exm}\label{ex:x1}
We now consider the modular curves of the family $X_1(N)$.  We start with the
three curves $X_1(11), X_1(14), X_1(15)$ of genus $1$.  Each of these curves is
isogenous to the respective curve $X_0(N)$, but our results and methods are not
invariant under isogeny because of changes in the cusps.

For $X_1(11)$, there are $10$ cusps, of which $5$ are rational and $5$ are defined
over $\Q(\zeta_7)^+$.  If we reduce $X_1(11) \bmod p$ for $p = 2, 3, 5$, we find
that there are $5$ points, which are necessarily the images of the $5$ cusps.  
Thus Theorem~\ref{thm:weaker} applies to these three cases: we conclude that if
$p \in \{2,3,5\}$ and $K$ is a number field in which every prime above $p$ has
inertial degree $1$, then every elliptic curve over $K$ with a point of order $11$
has bad reduction at all primes above $p$.  The stronger Theorem~\ref{thm:stronger}
applies only for $p = 2$.

The next example is $X_1(14)$, which has $12$ cusps, of which $6$ are rational
and the others fall into two orbits of points defined over ${\Q(\zeta_7)}^+$.
Since an elliptic curve over $\F_3$ or $\F_5$ with a subgroup of order $6$ has no
more points, Theorem~\ref{thm:weaker} applies in both cases, and 
Theorem~\ref{thm:stronger} to $p = 3$.  Similarly, for $X_1(15)$ there are $16$
cusps of which only $4$ are rational.  It follows that Theorem~\ref{thm:stronger}
applies to $p = 2$.  For $p = 7$,  the order of $X_1(15)$ is $8$.  It so happens
that the two cusp orbits of size $2$ have rational points over $\Q(\sqrt{-3})$ and
therefore mod $7$, so every point is represented by a cusp and Theorem~\ref{thm:weaker}
applies.

We briefly remark on $X_1(13)$, which has $12$ cusps of which $6$ are rational and
the other $6$ form a single orbit whose elements are defined over ${\Q(\zeta_{13})}^+$.
For primes up to $11$ the curve has $6, 6, 6, 8, 12$ respectively (it is not possible
for an elliptic curve to have $6$ points over $\F_2$, but $X_1(13)$ has genus $2$).
A straightforward variant of Theorem~\ref{thm:stronger} applies to $2, 3$ and a similar variation on Theorem~\ref{thm:weaker} would apply to $5$.  
\end{exm}
\begin{thm}\label{mostgeneral}
    Let $\Gamma$ be a modular curve of genus $1$ over $\Q$, and let $p$ be a prime.  Let $k$ be a number field.  Let $\{P_1,\ldots,P_r\}$ be a set of generators of the Mordell-Weil group of $\Gamma$ over $k$.  The following are equivalent:
    \begin{enumerate}
        \item For every prime $\p$ of $k$ lying over $p$, there is a cusp $Q$ of $\Gamma$ such that each $P_i$ is congruent to $Q$ modulo $\p$.
        \item For every prime $\p$ of $k$ lying over $p$, there is a cusp $Q$ such that every $k$-rational point $P$ is congruent to $Q$ modulo $\p$.
    \end{enumerate}
\end{thm}

\begin{proof} The implication $(b)\implies(a)$ is trivial.  The reverse implication is merely easy. \end{proof}

\begin{rem}\label{cuspsandtorsion}
    If $\Gamma$ is a modular curve of genus $1$, and if $\Gamma$ has a rational torsion point $T$ that is not a cusp, then Theorem \ref{mostgeneral} cannot apply because $T$ will be integral with respect to the cusps, which are all torsion.
\end{rem}


\begin{rem}\label{numberofcusps}
    A modular curve of genus $1$ over $\Q$ cannot have more than 16 rational cusps by Mazur's Theorem.  Of course, it could have nonrational cusps that reduce to a rational point modulo $p$.
\end{rem}

\begin{rem}\label{highergenus}
    One can apply a straightforward variation on Theorem~\ref{mostgeneral} to modular curves of higher genus as well, but the number of rational points is much more limited, being a finite set for any given number field $k$.  
\end{rem}

\section{Examples of divisibility}
In this section we give specific examples for Theorem~\ref{thm:stronger}.  In particular,
we have already seen that the hypothesis of surjectivity mod $p$ is satisfied for $E = X_0(20)$,
$p = 7$.  We now give a condition for the map $E(\Q) \oplus E^\sigma(\Q) \to E(K)$ to be 
surjective.  To be precise, we prove the following:

\begin{thm}\label{thm:surjectivity}
Let $q$ be prime and let $K = \Q(\sqrt{q})$.
\begin{enumerate} 
\item If $q \equiv 13, 17 \bmod 20$, then the natural map 
$E(\Q) \oplus E^\sigma(\Q) \to E(K)$ is surjective.  If the $2$-part of $\Sha(E^\sigma)$ is finite, 
then $E(K)$ has positive rank.
\item If $q \equiv 11, 19 \bmod 20$ and the $2$-part of $\Sha(E^\sigma)$ 
is finite, then $E(\Q) \oplus E^\sigma(\Q) \to E(K)$ is not surjective.
\end{enumerate}
\end{thm}

Combining this result with Theorems~\ref{thm:weaker},\ref{thm:stronger}, and observing that
$\qr{-20}{q} = -1$ if and only if $q \equiv 11, 13, 17, 19 \bmod 20$, we obtain:
\begin{cor}\label{cor:example}
Let $q$ be a prime such that $\qr{-20}{q} = -1$ and either $\qr{-1}{q} = 1$ or
$\qr{q}{3} = 1$.  Then every elliptic
curve over $\Q(\sqrt{q})$ with a rational cyclic subgroup of order $20$ has bad reduction at $3$.
\end{cor}

\begin{proof}
    If $\qr{q}{3} = 1$ then Theorem~\ref{thm:weaker} applies; if $\qr{q}{3} = -1$ and
    $\qr{-1}{q} = 1$ then Theorem~\ref{thm:stronger} applies.
\end{proof}

Theorem~\ref{thm:surjectivity}
is not really a new result, being a typical application of the theory of $2$-descent
on an elliptic curve over $\Q$; we include it to show that
Theorem~\ref{thm:stronger} is not vacuous.  Similar results could be proved for the other modular
curves appearing in Table~\ref{table:x0}.

Before proving this theorem, we start by establishing some notation and basic results for
$2$-descent on $E^\sigma$.  Our point of view is very close to that of \cite{schaefer}.
Let $q$ be a prime: since $E$ is defined by $y^2 = x^3+x^2+4x+4$,
the twist $E^\sigma$ is defined by $y^2 = x^3 + qx^2 + 4q^2x + 4q^3$.  
We use the standard $2$-descent map: let 
$A$ be the \'etale algebra $\Q[x]/(x^3+qx^2+4q^2x+4q^3) = \Q(\theta)$ and define
the map $\phi: E^\sigma(\Q) \to A^\times/{A^\times}^2$ by $\phi((x_0,y_0)) = x_0-\theta$ for 
most points.  We will use the isomorphism $A \cong \Q \oplus \Q(i)$ to rewrite the map
as $\phi((x_0,y_0)) \to (x_0+q,x_0+2qi)$.  For the point at infinity we have $\phi(0) = (1,1)$,
and at $(-q,0)$, where the first component is not defined, the condition of square norm
requires that the image be $(5,-q+2qi)$.  Further, the image of the map is supported at
the primes of bad reduction $2, 5, q$.

\begin{propdefn} For $q$ a prime congruent to $13$ or $17 \bmod 20$, let $n_q$ be an element
of $\Z[i]$ of norm $q$ such that
$$N(n_q) = q, \quad n_q \equiv 1 \bmod 2, \quad n_q \equiv \pm 1 \bmod 2+i.$$
This element is unique up to sign, and therefore up to squares.
\end{propdefn}

\begin{proof}
    The elements of norm $q$ are of the form $\pm a \pm bi, \pm b \pm ai$ where $a^2+b^2 = q$.
    We take $b$ to be even: then only the $\pm a \pm bi$ are $1 \bmod 2$.  Further, since
    $(a+bi)(a-bi) = q \equiv 2, 3 \bmod 2+i$, exactly one of $a+bi$ and $a-bi$ is $\pm 1 \bmod 2+i$.
\end{proof}

\begin{lem}\label{lem:cover}
Let $(a,b+ci) \in \Q^\times \oplus \Q(i)^\times$.  Then $(a,b+ci)$ is in the image of the $2$-descent map
if and only if the projective closure $C_{a,b+ci}$ of the curve in $\A^3$ defined by
$$2bst+c(s^2-t^2) = 2q, \quad ar^2-b(s^2-t^2)+2cst = q$$
has rational points.
\end{lem}

\begin{proof}
The condition for $(a,b+ci)$ to be in the image is that there exist $x_0 \in \Q$ with
$x_0+q = ar^2$ and $x_0+2qi = (b+ci)(s+ti)^2$.  Expanding the second equation and eliminating
$x_0$ produces the equations indicated, while it is easily verified that there is a point at
infinity if and only if $(a,b+ci)$ is in the image of the $2$-torsion subgroup.
\end{proof}

\begin{dfn}\label{def:selmer}
The {\em Selmer group} of $E^\sigma$ (for multiplication by $2$) is the subgroup of
$(a,b+ci) \in \Q^\times \oplus \Q(i)^\times$ mod squares
for which $C_{a,b+ci}$ has points everywhere locally.
\end{dfn}

We now proceed to the proof of Theorem~\ref{thm:surjectivity}.
\begin{proof}
{\em First case: $q \equiv 13, 17 \bmod 20$.}  Since $q \equiv 1 \bmod 4$, the subgroup
of $A^\times/{A^\times}^2$ is generated by $(1,i),(2,1+i),(5,2+i),(1,5),(q,n_q),(1,q)$,
which will be denoted $g_1, \dots, g_6$.  We show that the Selmer group is generated by
$g_1g_3g_6, g_5g_6$.  Note that the image of the $2$-torsion point is $(5,-q+2qi) = g_1g_3g_6$,
so this is automatically in the Selmer group.

First we prove that $C_{a,b+ci}$ has $2$-adic points if and only if 
$(a,b+ci) \in \langle g_1g_3,g_4,g_5,g_6 \rangle$.  For $g_4$ there are points with $s = 1/2$,
and for $g_6$ there are points with $s = 1/2$ if $q \equiv 5 \bmod 8$ and with $s = 1$ if
$q \equiv 1 \bmod 8$.  Instead of $g_5$ we will use $g_5g_6$, giving a curve defined by
$$r^2-d(s^2-t^2)+2est-1 = 2dst+e(s^2-t^2)-2 = 0,$$ where $n_q = d+ei$.  A tedious but elementary
calculation shows that in all cases with $d+ei \equiv 1 \bmod 2$
there is a $2$-adic point either at infinity or with one
of the variables equal to one of $0, 1, 2, 3, 1/2, -1$.

To show ``only if", we prove that there are no $2$-adic points for $g_1,g_2,g_1g_2$, these
being the nontrivial coset representatives for the known group.
In all cases it suffices to check mod $4$.

Proceeding to $q=5$, we show that $C_{a,b+ci}$ has $5$-adic points if and only if
$(a,b+ci) \in \langle g_1g_6,g_2g_6,g_3,g_5g_6 \rangle$.
We already have $g_1g_3g_6$, so we only check $g_1g_6,g_2g_6, g_5g_6$.
In the first two cases, there are $5$-adic points with $s=t$, $t=1$.  In the last, reducing mod
$5$ gives a union of two conics (this depends on the condition $d+ei \equiv \pm1 \bmod 2+i)$.
The intersection points are singular, but there are still smooth points, and these lift to
$\Q_5$.

Again we show that the coset representatives $g_1,g_4,g_1g_4$ give curves that are not
locally solvable at $5$.  For $g_1$ it suffices to work mod $5$; for the other two, 
mod $25$ is enough.

Thus the $2$-Selmer group is contained in the intersection of these two groups, which is
$\langle g_1g_3g_6,g_5g_6 \rangle$.  As above $g_1g_3g_6$ is the image of the $2$-torsion point,
so it is in the Selmer group, and it suffices to show that the curve corresponding to $g_5g_6$
is locally solvable at $q$ and $\infty$.  There are certainly real points, since
$(q,qn_q)$ is a square in $\R \oplus \C$.  To see that there are $q$-adic points,
it suffices to observe that the reduction mod $q$ is the union of four lines.  Let
$i = d/e \bmod q$, where as before $n_q = d+ei$.  Then the lines are
$r = \pm \sqrt{1+2i}, s + it = \pm \sqrt{2/e}$,
and every point that is on exactly one of the lines will lift.  To see that $2/e$ is a
square mod $q$, we use the rational biquadratic reciprocity law: using $[]$ for the
biquadratic residue symbol, we have
$$\left[\frac{-4ei}{d+2ei}\right] = \left[\frac{d-2ei}{d+2ei}\right] = \left[\frac{-2}{d+2ei}\right]\cdot (-1)^{(d^2-1)/8}$$
where the second equality follows from \cite[Exercises 9.27--30]{ireland-rosen}.
Squaring both sides gives $\left(\frac{-4ei}{q}\right) = \left(\frac{-2}{q}\right)$.  However,
$(-4i)/(-2) = (1+i)^2$ is a square, so $\left(\frac eq\right)=1$.

This completes the proof that the $2$-Selmer group is $\langle g_1g_3g_6,g_5g_6 \rangle$.
Under the assumption that the $2$-part of $\Sha$ is finite, this implies that $E^\sigma$ has
rank $1$.  The cokernel of the map $E(\Q) \oplus E^\sigma(\Q) \to E(\Q(\sqrt{q}))$ is
$2$-torsion, so if it is nonzero then a point of $E^\sigma(\Q)$ must become a multiple of
order $2$ in $E(\Q(\sqrt{q}))$.  But this would imply that the map on $2$-Selmer groups is
not injective.

The argument in the second case is similar but simpler, because we do not have to consider the
factorization of $q$ in $\Z[i]$.  The result is that the Selmer group of $E^\sigma$ is
generated by $(5,-1+2i)$ and $(1,q)$.  Assuming that $\Sha[2^\infty]$ is finite, this implies
that there is a generator mapping to $(1,q)$.  Such a generator will be a multiple of $2$
in $E(\Q(\sqrt{q}))$, implying that the map is not surjective.
\end{proof}

\end{document}